\documentclass{article}
\usepackage{amsmath}
\usepackage{amsfonts}
\usepackage{amsthm}
\usepackage{latexsym}
\usepackage{amssymb}
\usepackage{verbatim}
\usepackage{enumerate}
\font\logic=msam10 at 10pt
\newcommand{\forces}{\mbox{\logic\char'015}}
\newcommand{\restrict}{\mbox{\logic\char'026}}

\def\undertilde#1{{\baselineskip=0pt\vtop
  {\hbox{$#1$}\hbox{$\scriptscriptstyle\sim$}}}{}}

\newcommand{\less}{\mathord{<}}

\newcommand{\undom}{\emph{Undominated}}
\newcommand{\doms}{\emph{Dominating} }
\newcommand{\undoms}{\emph{Undominated} }

\newtheorem{thrm}{Theorem}[section]

\newtheoremstyle{hdefinition}%
  {\topsep}%
  {\topsep}%
  {\upshape}
  {}%
  {\bfseries}%
  {.}
  { }%
  {\thmnumber{#2 }\thmname{#1}\thmnote{ \rm(#3)}}%
\newtheoremstyle{hclaim}%
  {\topsep}%
  {\topsep}%
  {\itshape}%
  {}%
  {\bfseries}%
  {.}
  { }%
  {\thmname{#1}\thmnote{ \rm#3}}%

\theoremstyle{hclaim}
\newtheorem*{claim*}{Claim}

\theoremstyle{hdefinition}
\newtheorem{df}[thrm]{Definition}

\newtheorem{ques}[thrm]{Question}

\theoremstyle{hclaim}
\newtheorem{claim}{Claim}

\newcommand{\unspl}{\emph{unsplit}}
\newcommand{\spl}{\emph{split}}
\newcommand{\unspls}{\emph{unsplit} }
\newcommand{\spls}{\emph{split} }

\begin{document}

\title{The stationary set splitting game\thanks{The work in this paper began during the
Set Theory and Analysis program at the Fields Institute in the Fall
of 2002. The first author is supported in part by NSF grant
DMS-0401603, and thanks Juris Stepr\={a}ns, Paul Szeptycki and
Tetsuya Ishiu for helpful conversations on this topic. The research
of the second author is supported by the United States-Israel
Binational Science Foundation. This is the second author's
publication 902. Some of the research in this paper was conducted
during a visit by the first author to Rutgers University, supported
by NSF grant DMS-0600940.}}

\author{Paul B. Larson \and Saharon
Shelah}

\pagenumbering{arabic}
\maketitle

\begin{abstract}
The \emph{stationary set splitting game} is a game of perfect
information of length $\omega_{1}$ between two players, \unspls and
\spl, in which \unspls chooses stationarily many countable ordinals
and \spls tries to continuously divide them into two stationary
pieces. We show that it is possible in ZFC to force a winning
strategy for either player, or for neither. This gives a new
counterexample to $\Sigma^{2}_{2}$ maximality with a predicate for
the nonstationary ideal on $\omega_{1}$, and an example of a
consistently undetermined game of length $\omega_{1}$ with payoff
definable in the second-order monadic logic of order. We also show
that the determinacy of the game is consistent with Martin's Axiom
but not Martin's Maximum.
\end{abstract}

\vspace{\baselineskip}

The \emph{stationary set splitting game} ($\mathcal{SG}$) is a game
of perfect information of length $\omega_{1}$ between two players,
\unspls and \spl. In each round $\alpha$, \unspls either accepts or
rejects $\alpha$. If \unspls accepts $\alpha$, then $\spl$ puts
$\alpha$ into one of two sets $A$ and $B$. If \unspls rejects
$\alpha$ then $\spl$ does nothing. After all $\omega_{1}$ many
rounds have been played, \spls wins if \unspls has not accepted
stationarily often, or if both of $A$ and $B$ are stationary.

In this note we prove that it is possible to force a winning
strategy for either player in $\mathcal{SG}$, or for neither, and we
also show that the determinacy of $\mathcal{SG}$ is consistent with
Martin's Axiom but not Martin's Maximum \cite{FMS}.
We also present two guessing principles, $\mathcal{C}_{s}$
(\emph{club for} \spl) and $\mathcal{D}_{u}$ (\emph{diamond for}
\unspl), which imply the existence of winning strategies for \spls
and \unspl, respectively (and are therefore incompatible; see
Theorems \ref{pos2} and \ref{pos3}). These principles may be of
independent interest.

\section{Winning strategies}

\subsection{Strategies for \spl}

A collection $\mathcal{X}$ of countable sets is \emph{stationary} if
for every function $F \colon [\bigcup\mathcal{X}]^{\less\omega} \to
\bigcup\mathcal{X}$ there is an element of $\mathcal{X}$ closed
under $F$. A set $\mathcal{X}$ of countable sets is \emph{projective
stationary} \cite{FJ} if for every stationary $S \subset \omega_{1}$
the set of $X \in \mathcal{X}$ with $X \cap \omega_{1} \in S$ is
stationary. We note that a partial order $P$ is said to be
\emph{proper} if forcing with $P$ preserves the stationarity (in the
sense above) of stationary sets from the ground model (see
\cite{Shf}).

The following statement holds in fine structural models such as $L$.
It is a strengthening of the principle $(+)$ used in \cite{L}.
Justin Moore has pointed out to us that his Mapping Reflection
Principle \cite{M} implies the failure of $(+)$. We note also that
in the statement of $(+)$, ``projective stationary" can be replaced
with ``club" without strengthening the statement. We do not know if
that is the case for $\mathcal{C}{+}$.

\begin{df} Let $\mathcal{C}+$ be the statement that there exists a
projective stationary set $\mathcal{X}$ consisting of countable
elementary substructures of $H(\aleph_{2})$ such that for all $X$,
$Y$ in $\mathcal{X}$ with $X \cap \omega_{1} = Y \cap \omega_{1}$,
either every for every club $C \subset \omega_{1}$ in $X$ there is a
club $D \subset \omega_{1}$ in $Y$ with $D \cap X \subset C \cap X$,
or for every for every club $D \subset \omega_{1}$ in $Y$ there is a
club $C \subset \omega_{1}$ in $X$ with $C \cap X \subset D \cap X$.
\end{df}

Given a partial run of $\mathcal{SG}$ of length $\alpha$, we let
$E_{\alpha}$ be the set of $\beta < \alpha$ accepted by \unspl, and
we let $A_{\alpha}$, $B_{\alpha}$ be the partition of $E_{\alpha}$
chosen by \spl.


\begin{thrm}\label{pos} If $\mathcal{C}+$ holds then $\spls$ has a
winning strategy in $\mathcal{SG}$.
\end{thrm}

\begin{proof}
Let $\mathcal{X}$ be a set of countable elementary submodels of
$H(\aleph_{2})$ witnessing $\mathcal{C}+$, and for each $\alpha <
\omega_{1}$ let $\mathcal{X}_{\alpha}$ be the set of $X \in
\mathcal{X}$ with $X \cap \omega_{1} = \alpha$. Let $Z$ be the set
of $\alpha < \omega_{1}$ such that $\mathcal{X}_{\alpha}$ is
nonempty (since $\mathcal{X}$ is projective stationary, this set
contains a club).

Play for $\spls$ as follows. In round $\alpha \in Z$, if \unspls
accepts $\alpha$, let $\mathcal{Y}_{\alpha}$ be the set of all $X
\in \mathcal{X}_{\alpha}$ such that $X$ contains a stationary subset
of $\omega_{1}$, $E_{X}$, such that $E_{X} \cap \alpha =
E_{\alpha}$. If $\mathcal{Y}_{\alpha} = \emptyset$, put $\alpha \in
A_{\alpha + 1}$. Otherwise, since every club subset of $\omega_{1}$
in every member of $\mathcal{Y}_{\alpha}$ intersects $E_{\alpha}$,
there cannot be two club subsets of $\omega_{1}$ in
$\bigcup\mathcal{Y}_{\alpha}$ , one disjoint from $A_{\alpha}$ and
one disjoint from $B_{\alpha}$, since some club subset of
$\omega_{1}$ in $\bigcup\mathcal{Y}_{\alpha}$ would be contained in
both of these clubs. If any member of $\mathcal{Y}_{\alpha}$
contains a club subset of $\omega_{1}$ disjoint from $A_{\alpha}$,
then put $\alpha$ in $A_{\alpha+1}$, and if any member of
$\mathcal{Y}_{\alpha}$ contains a club subset of $\omega_{1}$
disjoint from $B_{\alpha}$, then put $\alpha$ in $B_{\alpha+1}$. If
neither case holds, put $\alpha \in A_{\alpha + 1}$.


Let $E$ be the play by $\unspls$ in a run of $\mathcal{SG}$ where
$\spls$ has played by this strategy, and let $A$ and $B$ be the
corresponding play by $\spls$. Let $C$ be a club subset of
$\omega_{1}$ and supposing that $E$ is stationary, fix $X \in
\mathcal{X}$ containing $E$, $A$, $B$ and $C$ with $X \cap
\omega_{1} \in E \cap C$. Then if $A \cap C \cap X \cap \omega_{1} =
\emptyset$, then $X \cap \omega_{1} \in A \cap C$, and if $B \cap C
\cap X \cap \omega_{1} = \emptyset$, then $X \cap \omega_{1} \in B
\cap C$, which shows that $C$ does not witness that $\unspls$ won
this run of the game.
\end{proof}

The following fact, in conjunction with Theorem \ref{pos}, shows
that Martin's Axiom is consistent with the existence of a winning
strategy for \spl.

\begin{thrm} The statement $\mathcal{C}+$ is preserved by forcing with
c.c.c. partial orders.
\end{thrm}


\begin{proof} Let $P$ be a c.c.c. forcing and let
$\mathcal{X}$ witness $\mathcal{C}+$. Let $\gamma$ be a regular
cardinal greater than $\aleph_{2}$ and $2^{|P|}$. Let $G \subset P$
be a $V$-generic filter, and let $$\mathcal{X}[G] =  \{ X[G] \cap
H(\aleph_{2})^{V[G]} : X \prec H(\gamma)^{V}, X \cap
H(\aleph_{2})^{V} \in \mathcal{X} \}.$$ Since every club subset of
$\omega_{1}$ in $V[G]$ contains one in $V$, in order to show that
$\mathcal{X}[G]$ witnesses $\mathcal{C}+$ in $V[G]$, it suffices to
show that $\mathcal{X}[G]$ is projective stationary there. Fix  a
$P$-name $\rho$ for a function from
$[H(\aleph_{2})^{V[G]}]^{\less\omega}$ to $H(\aleph_{2})^{V[G]}$.
For any countable $X \prec H(\gamma)$ with $X \cap H(\aleph_{2}) \in
\mathcal{X}$ and $\rho \in X$, $X[G] \cap H(\aleph_{2})^{V[G]}$ is
in $\mathcal{X}[G]$ and closed under the realization of $\rho$. Fix
a $P$-name $\tau$ for a stationary subset of $\omega_{1}$ and a
condition $p \in P$. Let $S$ be the set of countable ordinals forced
to be in $\tau$ by some condition below $p$. Then exist a countable
$X \prec H(\gamma)$ with $X \cap H(\aleph_{2}) \in \mathcal{X}$, $X
\cap \omega_{1} \in S$ and $\rho \in X$ and a condition $q$ below
$p$ forcing that $X[\dot{G}] \cap \omega_{1}$ (where $\dot{G}$ is
the name for the generic filter) is in the realization of $\tau$. By
genericity, then, $\mathcal{X}[G]$ is projective stationary.
\end{proof}

We do not know how to force $\mathcal{C}+$, however, and use a
different principle to force the existence of a winning strategy for
\spl.

\begin{df} Let $\mathcal{C}_{s}$ be the statement that there exist
$c_{\alpha}$ ($\alpha < \omega_{1}$ limit) such that each
$c_{\alpha}$ is a sequence $\langle a^{\alpha}_{\beta} : \beta <
\gamma_{\alpha} \rangle$ (for some countable $\gamma_{\alpha}$) of
cofinal subsets of $\alpha$ of orderype $\omega$ and
\begin{itemize}
\item for all limit $\alpha < \omega_{1}$ and all $\beta < \beta' <
\gamma_{\alpha}$, $a^{\alpha}_{\beta'}\setminus a^{\alpha}_{\beta}$
is finite;
\item for every club $C \subset \omega_{1}$ and every stationary
$E \subset \omega_{1}$ there exists an $a^{\alpha}_{\beta}$ with
$\alpha \in E$ such that $a^{\alpha}_{\beta} \setminus C$ is finite
and $a^{\alpha}_{\beta} \cap E$ is infinite.
\end{itemize}
\end{df}

The principle $\mathcal{C}_{s}$ also holds in fine structural models
such as $L$. The winning strategy for \spls given by
$\mathcal{C}_{s}$ is very similar to the one given by
$\mathcal{C}+$.

\begin{thrm}\label{pos2} If $\mathcal{C}_{s}$ holds then $\spls$ has a
winning strategy in $\mathcal{SG}$.
\end{thrm}

\begin{proof}
Let $a^{\alpha}_{\beta}$ ($\alpha < \omega_{1}$ limit, $\beta <
\gamma_{\alpha}$) witness $\mathcal{C}_{s}$. Play for \spls as
follows. In round $\alpha$, $\alpha$ a limit, if \unspls has
accepted $\alpha$ and if some $a^{\alpha}_{\beta}$ intersects
$A_{\alpha}$ infinitely and $B_{\alpha}$ finitely, then put $\alpha$
in $B_{\alpha + 1}$. If some $a^{\alpha}_{\beta}$ intersects
$B_{\alpha}$ infinitely and $A_{\alpha}$ finitely, then put $\alpha$
in $A_{\alpha + 1}$. Since the $a^{\alpha}_{\beta}$'s $(\beta <
\gamma_{\alpha})$ are $\subset$-decreasing mod finite, both cases
cannot occur. If neither case occurs, put $\alpha$ in $A_{\alpha +
1}$.

Let $E$ be the play by $\unspls$ in a run of $\mathcal{SG}$ where
$\spls$ has played by this strategy, and let $A$ and $B$ be the
corresponding play by $\spls$. Let $C$ be a club subset of
$\omega_{1}$ and supposing that $E$ is stationary, fix
$a^{\alpha}_{\beta}$ with $\alpha \in E$ such that
$a^{\alpha}_{\beta} \setminus C$ is finite and $a^{\alpha}_{\beta}
\cap E$ is infinite. Then if $A \cap a^{\alpha}_{\beta}$ is finite,
then $\alpha \in A \cap C$, and if $B \cap a^{\alpha}_{\beta}$ is
finite, then $\alpha \in B \cap C$, which shows that $C$ does not
witness that $\unspls$ won this run of the game.
\end{proof}

A partial order $P$ is said to be \emph{strategically}
$\omega$-closed if there exists a function $f \colon P^{\less\omega}
\to \mathcal{P}(P)$ such that whenever $\langle p_{i} : i \leq
n\rangle$ is a finite descending sequence in $P$,  $f(\langle p_{i}
: i \leq n \rangle)$ is a dense subset below $p_{n}$ and, whenever
$\langle p_{i} : i < \omega \rangle$ is a descending sequence in $P$
such that for each $n$ there exists a $j$ with $$p_{j} \in f(\langle
p_{i} : i \leq n\rangle),$$ the sequence has a lower bound in $P$.
It is easy to see that strategic $\omega$-closure is equal to the
property that for every countable $X \prec H((2^{|P|})^{+})$ and
every $(X,P)$-generic filter $g$ contained in $X$ there is a
condition in $P$ extending $g$.

Let us say that a set $a$ \emph{captures} a pair $E,C$ if $a
\setminus C$ is finite and $a \cap E$ is infinite. Given $A \subset
\omega_{1}$, let $\mathbb{C}(A)$ be the partial order which adds a
club subset of $A$ by initial segments. We force $\mathcal{C}_{s}$
by first adding a potential $\mathcal{C}_{s}$-sequence by initial
segments, and then iterating to kill off every counterexample.

We refer the reader to \cite{Shf} for background on countable
support iterations of proper forcing.

\begin{thrm}\label{csforce} Suppose that CH and $2^{\aleph_{1}} = \aleph_{2}$ hold.
Let $\bar{P} = \langle P_{\eta}, \undertilde{Q}_{\eta} : \eta <
\omega_{2} \rangle$ be a countable support iteration such that
$P_{0}$ is the partial order consisting of sequences $\langle
c_{\alpha} : \alpha < \delta \text{ limit}\rangle$, for some
countable ordinal $\delta$, such that each $c_{\alpha}$ is a
sequence $\langle a^{\alpha}_{\beta} : \beta <
\gamma_{\alpha}\rangle$ (for some countable ordinal
$\gamma_{\alpha}$) of cofinal subsets of $\alpha$ of ordertype
$\omega$, deceasing by mod-finite inclusion (and $P_{0}$ is ordered
by extension). Suppose that the remainder of $\bar{P}$ satisfies the
following conditions.
\begin{itemize}
\item For each nonzero $\eta < \omega_{2}$ there is a $P_{\eta}$-name
$\tau_{\eta}$ for a subset of $\omega_{1}$ such that if
$(\tau_{\eta})_{G_{\eta}}$ (where $G_{\eta}$ is the restriction of
the generic filter to $P_{\eta}$) is stationary in the $P_{\eta}$
extension and there exists a club $C \subset \omega_{1}$ in this
extension such that no $a^{\alpha}_{\beta}$ with $\alpha \in
\tau_{G_{\eta}}$ captures the pair $\tau_{G_{\eta}}, C$, then
$\undertilde{Q}_{\eta}$ is $\mathbb{C}(\omega_{1} \setminus
(\tau_{\eta})_{G_{\eta}})$ (and otherwise, $\undertilde{Q}_{\eta}$
is $\mathbb{C}(\omega_{1})$).
\item For every pair $E,C$ of subsets of $\omega_{1}$ in any
$P_{\eta}$-extension $(\eta < \omega_{2})$, if $E$ is stationary in
this extension and $C$ is club and no $a^{\alpha}_{\beta}$ with
$\alpha \in E$ captures $E,C$, then there is a $\rho \in [\eta,
\omega_{2})$ such that if $E$ is stationary in the $P_{\rho}$
extension, then $\undertilde{Q}_{\rho}$ is $\mathbb{C}(\omega_{1}
\setminus E)$.
\end{itemize}
Then $\bar{P}$ is strategically $\omega$-closed, and
$\mathcal{C}_{s}$ holds in the $\bar{P}$-extension. Furthermore, in
the $\bar{P}$ extension, $\Diamond(S)$ holds for every stationary $S
\subset \omega_{1}$.
\end{thrm}

\begin{proof} Let $X$ be a countable elementary submodel of
$H((2^{|\bar{P}|)^+})$ with $\bar{P} \in X$, let $g$ be an
$X$-generic filter contained in $\bar{P} \cap X$. Let $\gamma_{X
\cap \omega_{1}}$ be the ordertype of $X \cap \omega_{2}$, and for
each $\beta < \gamma_{X \cap \omega_{1}}$, let $\eta_{\beta}$ be the
$\beta$th member of $X \cap \omega_{2}$. For each $\beta < \gamma_{X
\cap \omega_{1}}$, let $a^{X \cap \omega_{1}}_{\beta}$ be a cofinal
subset of $X \cap \omega_{1}$ of ordertype $\omega$ such that,
letting $g_{\eta}$ denote the restriction of $g$ to $P_{\eta}$,
\begin{itemize}
\item for all $\beta' < \beta < \gamma_{X \cap \omega_{1}}$, $a^{X \cap \omega_{1}}_{\beta}
\setminus a^{X \cap \omega_{1}}_{\beta'}$ is finite;
\item $a^{\alpha}_{\beta}$ is eventually contained in every club
subset of $\omega_{1}$ in $X[g_{\eta_{\beta}}]$ and intersects
infinitely every stationary subset of $\omega_{1}$ in every
$X[g_{\eta_{\beta'}}]$, $\beta' \in [\beta, \gamma_{X \cap
\omega_{1}})$.
\end{itemize}


It remains to see that we can extend $g$ to a condition whose first
coordinate is given by adding $c_{X \cap \omega_{1}} = \langle
a^{\alpha}_{\beta} : \beta < \gamma_{X \cap \omega_{1}}\rangle$ to
the union of the first coordinates of the elements of $g$, and whose
$\eta$th coordinate, for each nonzero $\eta \in X \cap \omega_{2}$
is the condition given by the union of $\{ X \cap \omega_{1} \}$ and
the set of realizations of the $\eta$th coordinates of the members
of $g$.
We do this by induction on $\eta$,
letting $g'_{\eta}$ be our extended condition in $P_{\eta}$.


For each $\eta \in \omega_{2} \cap X$, there is a $P_{\eta}$-name
$\sigma \in X$ for a club subset of $\omega_{1}$ such that if, in
the $P_{\eta}$-extension $(\tau_{\eta})_{G_{\eta}}$ is stationary
and there exists a club $C$ such that $\tau_{G_{\eta}}, C$ is not
captured by any $a^{\alpha}_{\beta}$ with $\alpha \in
(\tau_{\eta})_{G_{\eta}}$, then $\sigma_{G_{\eta}}$ is such a $C$.
However, the realizations of $\tau_{\eta}$ and $\sigma$ by $g$ are
captured by $a^{X \cap \omega_{1}}_{o.t.(\eta \cap \omega_{2})}$, so
$g'_{\eta}$ forces that $\tau_{G_{\eta}},\sigma_{G_{\eta}}$ is
captured by $a^{X \cap \omega_{1}}_{o.t.(\eta \cap \omega_{2})}$. It
follows that $g'_{\eta}$ forces that either $\undertilde{Q}_{\eta}$
is $\mathbb{C}(\omega_{1})$, or $X \cap \omega_{1}$ is not in
$\tau_{G_{\eta}}$. In either case, the union of the members of $g
\cap \undertilde{Q}_{\eta}$ be can extended to a condition in
$\undertilde{Q}_{\eta}$ by adding $\{X \cap \omega_{1}\}$.

To see that $\Diamond(S)$ holds for every stationary $S \subset
\omega_{1}$ in the $\bar{P}$ extension, fix such an $S$ in the
$P_{\alpha}$ extension for some $\alpha < \omega_{2}$. Since
$\bar{P}$ is $(\omega, \infty)$ distributive, there exists in this
extension a set $\langle e^{\delta}_{\beta} : \delta, \beta <
\omega_{1} \rangle$ such that for every $\delta < \omega_{1}$ and
every $x \subset \delta$ there are uncountably many $\beta$ such
that $e^{\delta}_{\beta} = x$. Then, letting $T \in
\mathcal{P}(\omega_{1})^{V[G_{\alpha}]}$ be the set such that the
realization of $\undertilde{Q}_{\alpha}$ is $\mathbb{C}(T)$,
$\undertilde{Q}_{\alpha}$ adds a $\Diamond$ sequence $\langle
b_{\delta} : \delta \in S \rangle$ defined by letting $b_{\delta}$
be $e^{\delta}_{\beta}$, where the $\beta$th element of $T$ above
$\beta$ is the first element of the generic club for
$\undertilde{Q}_{\alpha}$ above $\delta$. To see that this is a
$\Diamond$ sequence, note that since $S$ is stationary in the
$\bar{P}$ extension, there are stationarily many elementary
submodels $X$ of any sufficiently large $H(\theta)^{V[G]}$ in this
extension with $X \cap \omega_{1} \in S$. Then $X \cap
(G/G_{\alpha})$ is a $(X \cap V[G_{\alpha}],
\bar{P}/P_{\alpha})$-generic filter which can be extended to a
condition in $\bar{P}/P_{\alpha}$ by adding $X \cap \omega_{1}$ to
each coordinate, and extended again to make any element of $T
\setminus ((X \cap \omega_{1}) + 1)$ the least element of the
generic club for $\undertilde{Q}_{\alpha}$ above $X \cap
\omega_{1}$. That $\langle b_{\beta} : \beta \in S \rangle$ is a
$\Diamond$ sequence then follows by genericity.
\end{proof}

Section \ref{indet} shows that proper forcing does not always
preserve the existence of a winning strategy for \spl.

\subsection{A strategy for \unspl}

In this section we show that it is consistent for \unspls to have a
winning strategy in $\mathcal{SG}$. We do this via the following
guessing principle.

\begin{df} Let $\mathcal{D}_{u}$ be the statement that there exists
a diamond sequence $\langle \sigma_{\alpha} : \alpha < \omega_{1}
\rangle$ such that for every $E \subset \omega_{1}$ there is a club
$C \subset \omega_{1}$ such that either $$\forall \alpha \in C ((E
\cap \alpha = \sigma_{\alpha}) \Rightarrow \alpha \in E)$$ or
$$\forall \alpha \in C ((E
\cap \alpha = \sigma_{\alpha}) \Rightarrow \alpha \not\in E).$$
\end{df}

\begin{thrm}\label{pos3} If $\mathcal{D}_{u}$ holds then \unspls has a winning
strategy in $\mathcal{SG}$.
\end{thrm}

\begin{proof} Let $\langle \sigma_{\alpha} : \alpha < \omega_{1} \rangle$
witness $\mathcal{D}_{u}$. Play for \unspls by accepting $\alpha$ if
and only if $\sigma_{\alpha} = A_{\alpha}$. At the end of the game,
the set of $\alpha$ such that $\sigma_{\alpha} = A_{\alpha}$ is
stationary, and there is a club $C$ such that either for all
$\alpha$ in $C$, if $\sigma_{\alpha} = A_{\alpha}$, then $\alpha$ is
in $A$, or for all $\alpha$ in $C$, if $\sigma_{\alpha} =
A_{\alpha}$, then $\alpha$ is in $B$. In either case, \spls has
lost.
\end{proof}

Our iteration to force $\mathcal{D}_{u}$ employs the same strategy
as the iteration for $\mathcal{C}_{s}$. We first force to add a
$\Diamond$-sequence $\langle \sigma_{\alpha} : \alpha < \omega_{1}
\rangle$ by initial segments, and
we then iterate to make this sequence witness $\mathcal{D}_{u}$,
iteratively forcing a club through the set of $\alpha < \omega_{1}$
such that $\sigma_{\alpha} \neq E \cap \alpha$ or $\alpha \in E$ for
each $E \subset \omega_{1}$ such that the sets $\{ \alpha \in E \mid
\sigma_{\alpha} = E \cap \alpha \}$ and $\{ \alpha \in \omega_{1}
\setminus E \mid \sigma_{\alpha} = E \cap \alpha \}$ are both
stationary.

More specifically, we have the following. Given a sequence $\Sigma =
\langle \sigma_{\alpha} : \alpha < \omega_{1} \rangle$ such that
each $\sigma_{\alpha}$ is a subset of $\alpha$, and given $E \subset
\omega_{1}$, let $A(\Sigma, E)$ be the set of $\alpha \in E$ such
that $\sigma_{\alpha} = E \cap \alpha$, and let $B(\Sigma, E)$ be
the set of $\alpha \in  \omega_{1} \setminus E$ such that
$\sigma_{\alpha} = E \cap \alpha$.

\begin{thrm}\label{duforce} Suppose that CH + $2^{\aleph_{1}} = \aleph_{2}$ holds,
and let $\bar{P}$ be a countable support iteration $\langle
P_{\alpha}, \undertilde{Q}_{\alpha} : \alpha < \omega_{2} \rangle$
such that $P_{0}$ is the partial order consisting of sequences
$\langle \sigma_{\beta} : \beta < \gamma\rangle$, for some countable
ordinal $\gamma$, such that each $\sigma_{\beta}$ is a subset of
$\beta$, ordered by extension. Let $\Sigma$ be the sequence added by
$P_{0}$ and suppose that the remainder of $\bar{P}$ satisfies the
following conditions.
\begin{itemize}
\item Each $\undertilde{Q}_{\alpha}$ is either $\mathbb{C}(\omega_{1})$
or $\mathbb{C}(\omega_{1} \setminus B(\Sigma, E))$ for some $E
\subset \omega_{1}$ such that $A(\Sigma, E)$ and $B(\Sigma, E)$ are
both stationary.
\item For every $E \subset \omega_{1}$ in any
$P_{\alpha}$-extension $(\alpha < \omega_{2})$ there is a $\gamma
\in [\alpha, \omega_{2})$ such that if $A(\Sigma, E)$ and $B(\Sigma,
E)$ are both stationary in the $P_{\gamma}$ extension, then
$\undertilde{Q}_{\gamma}$ is $\mathbb{C}(\omega_{1} \setminus
B(\Sigma, E))$.
\end{itemize}
Then $\bar{P}$ is strategically $\omega$-closed, and in the
$\bar{P}$-extension, $\mathcal{D}_{u}$ holds. Furthermore, in the
$\bar{P}$ extension, $\Diamond(S)$ holds for every stationary $S
\subset \omega_{1}$.
\end{thrm}

\begin{proof} The iteration $\bar{P}$ is clearly strategically $\omega$-closed,
since for any countable $X \prec
H((2^{|\bar{P}|})^{+})$ and any $(X, \bar{P})$-generic filter $g$
contained in $X$, one can extend $g$ to a condition by making
$\sigma_{X \cap \omega_{1}}$ unequal to the realization by $g$ of
any name in $X$ for a subset of $\omega_{1}$, and adding $X \cap
\omega_{1}$ to all the clubs being added by the
$\undertilde{Q}_{\alpha}$'s, $\alpha \in X \cap \omega_{2}$. It is
clear also that in the $\bar{P}$-extension there is no $E \subset
\omega_{1}$ such that $A(\Sigma, E)$ and $B(\Sigma, E)$ are both
stationary.

To see that at least one of $A(\Sigma,E)$ and $B(\Sigma, E)$ is
stationary for each $E \subset \omega_{1}$, we first note the
following.

\begin{claim}[1] Suppose that $E \subset \omega_{1}$ is a member
of the $P_{\alpha}$ extension, for some $\alpha < \omega_{2}$, and
$A(\Sigma, E)$ is stationary in this extension. Then $A(\Sigma, E)$
remains stationary in the $\bar{P}$ extension.
\end{claim}

Note that $A(\Sigma, E)$ has countable intersection with $B(\Sigma,
F)$, for every $F \subset \omega_{1}$. Fix $X \prec
H(((2^{|\bar{P}|})^{+})^{V})^{V[G_{\alpha}]}$ (where $G_{\alpha}$ is
the restriction of the generic filter $G$ to $P_{\alpha}$) with $X
\cap \omega_{1} \in A(\Sigma, E)$ and $A(\Sigma, E) \in X$. Then any
$(X,\bar{P}/P_{\alpha})$-generic filter contained in $X$ can be
extended to a condition by adding $X \cap \omega_{1}$ to the clubs
being added at every stage of $\bar{P}$ after the first.

Similar reasoning shows the following two facts, which complete the
proof that $\Sigma$ witnesses $\mathcal{D}_{u}$ in the $\bar{P}$
extension.

\begin{claim}[2] Suppose that $E\subset \omega_{1}$ is a member
of the $P_{\alpha}$ extension, for some $\alpha < \omega_{2}$, and
not a member of the $P_{\gamma}$ extension, for any $\gamma <
\alpha$. Then $A(\Sigma, E) \cup B(\Sigma, E)$ is stationary in the
$P_{\alpha}$ extension.
\end{claim}

To see Claim 2, let $\tau$ be a $P_{\alpha}$-name for a subset of
$\omega_{1}$ which is forced to be unequal to any such subset in any
$P_{\gamma}$ extension, for any $\gamma < \alpha$. Fix $X \prec
H(((2^{|\bar{P}|})^{+}))^{V}$  with $\tau \in X$. Let $g$ be an
$(X,P_{\alpha})$-generic filter, and note that the realization of
$\tau \restrict (X \cap \omega_{1})$ by $g$ is different from the
realizations of $\rho \restrict (X \cap \omega_{1})$ by $g$ for any
$P_{\gamma}$-name $\rho \in X$ for a subset of $\omega_{1}$, for any
$\gamma \in X \cap \alpha$. It follows that adding the realization
of $\tau \restrict (X \cap \omega_{1})$ by $g$ to the union of the
first coordinate projection of $g$ gives a condition in $P_{0}$
forcing that $X \cap \omega_{1}$ is not in any $\Sigma(B,
\rho_{G_{\gamma}})$, for any for any $P_{\gamma}$-name $\rho \in X$
for a subset of $\omega_{1}$, for any $\gamma \in X \cap \alpha$.
Therefore, we can add $X \cap \omega_{1}$ to the clubs being added
in every other stage of $\bar{P}$ in $X \cap \alpha$, and get a
condition extending every condition in $g$.

\begin{claim}[3] Suppose that $E \subset \omega_{1}$ is a member
of the $P_{\alpha}$ extension, for some $\alpha < \omega_{2}$, and
$A(\Sigma, E)$ is nonstationary in this extension. Then $B(\Sigma,
E)$ remains stationary in the $\bar{P}$ extension.
\end{claim}

This is similar to the previous claims, noting that every subsequent
stage of $\bar{P}$ forces a club though the complement of a set with
countable intersection with $B(\Sigma, E)$.


The proof that $\Diamond(S)$ holds for every stationary $S \subset
\omega_{1}$ in the $\bar{P}$ extension is (literally) the same as in
the proof of Theorem \ref{csforce}.
\end{proof}

Note that that the iterations $\bar{P}$ in Theorems \ref{csforce}
and \ref{duforce} are strategically $\omega$-closed.

\subsection{$\Sigma^{2}_{2}$ maximality}

The statements that \spls and \unspls have winning strategies in
$\mathcal{SG}$ are each $\Sigma^{2}_{2}$ in a predicate for
$NS_{\omega_{1}}$, and they are obviously not consistent with each
other. Woodin (see \cite{KLZ}) has shown that if there is a proper
class of measurable Woodin cardinals, then there exists in a forcing
extension a transitive class model of ZFC satisfying all
$\Sigma^{2}_{2}$ sentences $\phi$ such that $\phi$ + CH can be
forced over the ground model. The results here show that this result
cannot be extended to include a predicate for $NS_{\omega_{1}}$.
This was known already, in that $\Diamond^{*}$ (in the sense of
\cite{K}) and ``the restriction of $NS_{\omega_{1}}$ to some
stationary set is $\aleph_{1}$ dense" were both known to be
consistent with $\Diamond$ (the second of these is due to Woodin,
uses large cardinals and is unpublished, though a related proof,
also due to Woodin, appears in \cite{F}). Our example is simpler and
doesn't use large cardinals; it also gives (we believe, for the
first time) a counterexample consisting of two sentences each
consistent with ``$\Diamond(S)$ holds for every stationary set $S
\subset \omega_{1}$."

\subsection{A determined variation}

There are many natural variations of $\mathcal{SG}$. We show that
one such variation is determined.

\begin{thrm} Let $\mathcal{G}$ be the following game of length
$\omega_{1}$. In round $\alpha$, player $I$ puts $\alpha$ into one
of two sets $E_{0}$ and $E_{1}$, and player $II$ puts $\alpha$ into
one of two sets $A_{0}$ and $A_{1}$. After all $\omega_{1}$ rounds
have been played, $II$ wins if one of the following pairs of set are
both stationary.
\begin{itemize}
\item $E_{0} \cap A_{0}$ and $E_{0} \cap A_{1}$
\item $E_{1} \cap A_{0}$ and $E_{1} \cap A_{1}$
\end{itemize}
Then $II$ has a winning strategy in $\mathcal{G}$.
\end{thrm}

\begin{proof}
Let $B_{00}$, $B_{01}$, $B_{10}$ and $B_{11}$ be pairwise disjoint
stationary subsets of $\omega_{1}$.
In round $\alpha$, if $\alpha$ is in $B_{ij}$, let $II$ put $\alpha$
in $A_{i}$ if $I$ put $\alpha$ in $E_{0}$ and in $A_{j}$ otherwise.
Then after all $\omega_{1}$ many rounds have been played, suppose
that $A_{i} \cap E_{0}$ is nonstationary. Then $B_{i0}$ and $B_{i1}$
are both contained in $E_{1}$ modulo $NS_{\omega_{1}}$, which means
that $E_{1} \cap A_{0}$ and $E_{1} \cap A_{1}$ are both stationary.
Similarly, if $A_{i} \cap E_{1}$ is nonstationary then $B_{0i}$ and
$B_{1i}$ are both contained in $E_{0}$ modulo $NS_{\omega_{1}}$,
which means that $E_{0} \cap A_{0}$ and $E_{0} \cap A_{1}$ are both
stationary.
\end{proof}


\section{Indeterminacy from forcing axioms}\label{indet}

The axiom PFA$^{+2}$ says that whenever $P$ is a proper partial
order, $D_{\alpha}$ $(\alpha < \omega_{1})$ are dense subsets of $P$
and $\sigma_{1}$, $\sigma_{2}$ are $P$-names for stationary subsets
of $\omega_{1}$, there is a filter $G \subset P$ such that $G \cap
D_{\alpha} \neq \emptyset$ for each $\alpha < \omega_{1}$, and such
that $\{ \alpha < \omega_{1} \mid \exists p \in G\text{ }p \forces
\check{\alpha} \in \sigma_{i}\}$ is stationary for each $i \in
\{1,2\}$. Theorems \ref{csforce} and \ref{duforce} together show
that PFA$^{+2}$ implies the indeterminacy of $\mathcal{SG}$.
Furthermore, a straightforward argument shows that the following
statement implies the nonexistence of a winning strategy for \unspls
in $\mathcal{SG}$, where $Add(1, \omega_{1})$ is the partial order
that adds a subset of $\omega_{1}$ by initial segments : for any
pair $\sigma_{1}$,$\sigma_{2}$ of $Add(1, \omega_{1})$-names for
stationary subsets of $\omega_{1}$, there is a filter $G \subset
Add(1, \omega_{1})$ realizing both $\sigma_{1}$ and $\sigma_{2}$ as
stationary sets. This statement is trivially subsumed by PFA$^{+2}$,
but also holds in the collapse of a sufficiently large cardinal to
be $\omega_{2}$, and thus is consistent with CH.

The axiom Martin's Maximum \cite{FMS} says that whenever $P$ is a
partial order such that forcing with $P$ preserves stationary
subsets of $\omega_{1}$ and $D_{\alpha}$ $(\alpha < \omega_{1})$ are
dense subsets of $P$, there is a filter $G \subset P$ such that $G
\cap D_{\alpha} \neq \emptyset$ for each $\alpha < \omega_{1}$.

\begin{thrm}\label{mmundet} Martin's Maximum implies that $\mathcal{SG}$ is
undetermined.
\end{thrm}

\begin{proof}
Fix a strategy $\Sigma$ for $\unspls$ in $\mathcal{SG}$, and let
$E$, $A$, and $B$ be the result of a generic run of $\mathcal{SG}$
where $\unspls$ plays by $\Sigma$ (the partial order consists of
countable partial plays where $\unspls$ plays by $\Sigma$, ordered
by extension). If the complement of $E$ has stationary intersection
with every stationary subset of $\omega_{1}$ in the ground model,
one can force to kill the stationarity of $E$ in such a way that the
induced two step forcing preserves stationary subsets of
$\omega_{1}$ and produces a run of $\mathcal{SG}$ where $\unspls$
plays by $\Sigma$ and loses. If the complement of $E$ does not have
stationary intersection with some stationary $F \subset \omega_{1}$
in the ground model, then there is a partial run of the game $p$ and
a name $\tau$ for a club such that $p$ forces that $E$ will contain
$F \cap \tau_{G}$. Then there exists in the ground model a run of
$\mathcal{SG}$ extending $p$ in which $\unspls$ plays by $\Sigma$
and loses: $\spls$ picks a pair of disjoint stationary subsets
$F_{0}$, $F_{1}$ of $F$, and plays so that
\begin{itemize}
\item for every $\alpha < \omega_{1}$, some initial segment of the
play forces some ordinal greater than $\alpha$ to be in $\tau$,
\item whenever $\unspls$ accepts $\alpha \in F$, $\spls$ puts
$\alpha$ in $A$ if $\alpha \in F_{0}$ and puts $\alpha \in B$ if
$\alpha \in F_{1}$.
\end{itemize}

Now fix a strategy $\Sigma$ for $\spls$ in $\mathcal{SG}$, and
generically add a regressive function $f$ on $\omega_{1}$ by initial
segments. Let $E^{\alpha} = f^{-1}(\alpha)$ and let $A^{\alpha},
B^{\alpha}$ be the responses given by $\Sigma$ to a play of
$E^{\alpha}$ by \unspl. Note that each $E^{\alpha}$ will be
stationary.

Suppose that there exist an $\alpha < \omega_{1}$ and stationary
sets $S$, $T$ in the ground model such that  $(S \cap E^{\alpha})
\setminus A^{\alpha}$ and $(T \cap E^{\alpha}) \setminus B^{\alpha}$
are both nonstationary. Then there is a condition $p$ in our forcing
(i.e., a regressive function on some countable ordinal) such that
$p$ forces that $(S\cap E^{\alpha}) \subset A^{\alpha}$ and $(T \cap
E^{\alpha}) \subset B^{\alpha}$, modulo nonstationarity (and so in
particular $S$ and $T$ have nonstationary intersection). Let $\tau$
be a name for a club disjoint from $(S \cap E^{\alpha}) \setminus
A^{\alpha}$ and $(T \cap E^{\alpha}) \setminus B^{\alpha}$. Extend
$p$ to a filter $f$ (identified with the corresponding function)
realizing $\tau$ as a club subset of $\omega_{1}$, at successor
stages extending to add a new element to the realization of $\tau$,
and at limit stages (when for some $\beta < \omega$, $f \restrict
\beta$ has been decided and $f(\beta)$ has not, and $\beta$ is
forced by $f \restrict \beta$ to be a limit member of the
realization of $\tau$) extending so that $f(\beta) = \alpha$ if and
only if $\beta \in S$.
Then the run of $\mathcal{SG}$ corresponding to $f^{-1}(\alpha)$ is
winning for \unspl, since the corresponding set $B^{\alpha}$ is
nonstationary.

If there exist no such $\alpha$, $S$, $T$, there is a function $h$
on $\omega_{1}$ such that each $h(\alpha) \in \{A^{\alpha},
B^{\alpha}\}$ and the forcing to shoot a club through the set of
$\beta$ such that $f(\beta) = \alpha \Rightarrow \beta \in
h(\alpha)$ preserves stationary subsets of the ground model. Then
Martin's Maximum applied to the corresponding two step forcing
produces a run of $\mathcal{SG}$ (the run for any $f^{-1}(\alpha)$
which is stationary) where $\spls$ plays by $\Sigma$ and loses.
\end{proof}





Theorem \ref{mmundet} leads to the following question.

\begin{ques} Does the Proper Forcing Axiom imply that $\mathcal{SG}$
is not determined?
\end{ques}

The following question is also interesting. The consistency of the
$\aleph_{1}$-density of $NS_{\omega_{1}}$ (relative to the
consistency of AD$^{L(\mathbb{R})}$) is shown in \cite{W99}.

\begin{ques} Does the $\aleph_{1}$-density of $NS_{\omega_{1}}$
decide the determinacy of $\mathcal{SG}$?
\end{ques}

\section{MLO games}

The second-order Monadic Logic of Order (MLO) is an extension of
first-order logic with logical constants $=$, $\in$ and $\subset$
and a binary symbol $<$ as the only non-logical constant, allowing
quantification over subsets of the domain. Every ordinal is a model
for MLO, interpreting $<$ as $\in$.

Given an ordinal $\alpha$, an MLO game of length $\alpha$ is
determined by an MLO formula $\phi$ with two free variables for
subsets of the domain. In such a game, two players each build a
subset of $\alpha$, and the winner is determined by whether these
two sets satisfy the formula in $\alpha$.

B\"{u}chi and Landweber \cite{BL69} proved the determinacy of all
MLO games of length $\omega$. Recently, Shomrat \cite{Sho} extended
this result to games of length less than $\omega^{\omega}$, and
Rabinovich \cite{R} extended it further to all MLO games of
countable length. The stationary set splitting game is an example of
an MLO game of length $\omega_{1}$ whose determinacy is independent
of ZFC.

We thank Assaf Rinot for pointing out to us the connection between
$\mathcal{SG}$ and MLO games.

\noindent Department of Mathematics and Statistics,
Miami University, Oxford, Ohio 45056, USA; Email: {\tt
larsonpb@muohio.edu} \vspace{\baselineskip}

\noindent The Hebrew University of Jerusalem, Einstein Institute of
Mathematics,\\ Edmond J. Safra Campus, Givat Ram, Jersalem 91904,
Israel \vspace{\baselineskip}

\noindent Department of Mathematics, Hill Center - Busch
Campus, Rutgers, The State University of New Jersey, 110
Frelinghuysen Road, Piscataway, NJ 08854-8019, USA; Email: {\tt
shelah@math.huji.ac.il}

\end{document}